\documentclass{amsart}
\usepackage{amsfonts,amssymb,amsmath,amsthm}
\usepackage{url}
\usepackage{enumerate}

\newtheorem{thrm}{Theorem}[section]

\newtheorem{prop}[thrm]{Proposition}

\theoremstyle{definition}
\newtheorem{definition}[thrm]{Definition}

\newtheorem{ex}[thrm]{Example}
\numberwithin{equation}{section}

\author[S. Ianu\c{s}]{S. Ianu\c{s}$^*$}
\address{
Faculty of Mathematics and Computer Science\\
University of Bucharest\\
Str. Academiei 14, Bucharest 70109, Romania}
\thanks{$^*$Passed away on April 8th, 2010}

\author{A.~M. Ionescu}
\address{
Mathematics Department\\
Politehnica University of Bucharest\\
Splaiul Independen\c{t}ei 313, Sector 6, Bucharest 060042, Romania}
\email{aionescu@math.pub.ro}

\author{R. Mocanu}
\email{xipita@yahoo.com}

\author{G.~E. V\^{\i}lcu}
\address{
Department of Mathematics and Computer Science\\
Petroleum-Gas University of Ploie\c sti\\
Bulevardul Bucure\c sti, Nr. 39, Ploie\c sti 100680, Romania and
Research Center in Geometry, Topology and Algebra\\
University of Bucharest\\
Str. Academiei 14, Bucharest 70109, Romania}
 \email{gvilcu@upg-ploiesti.ro}


\begin{document}

\title[Riemannian submersions from almost contact metric manifolds]{Riemannian submersions from almost contact metric manifolds}

\maketitle

\begin{abstract}
In this paper we obtain the structure equation of a contact-complex
Riemannian submersion and give some applications of this equation in
the study of almost cosymplectic manifolds with K\"{a}hler fibres.
\end{abstract}

\section{Introduction} \label{sect1}

The Riemannian submersions are of great interest both in mathematics
and physics, owing to their applications in the Yang-Mills theory,
Kaluza-Klein theory, supergravity and superstring theories
\cite{BOU,BLW,CORM,IV,IV2,MST,VIS,WATS}, they being intensively
studied for different ambient spaces by several authors (see e.g.
\cite{BI,CC,CHEN,CHN4,ESC,FAL,IMaV,IMV,NAR,SAH,TM,VW,VLC2,VLC,WAT1,WV2}).
In this paper we study the so-called contact-complex Riemannian
submersions, i.e. Riemannian submersions from an almost contact
metric manifold $(M, \varphi, \xi, \eta,g)$ onto an almost Hermitian
manifold $(N,J,g')$, which are $(\varphi,J)$-holomorphic mappings
(see \cite{IP}). The paper is organized as follows. The second
section contains some preliminaries on almost contact metric
manifolds and Riemannian submersions. In the next section, some
results concerning the contact-complex submersions are given; among
other, we discuss the transference of Gray-type curvature conditions
from the total space to the fibres and to the base space. In the 4th
section we obtain the structure equation of a contact-complex
Riemannian submersion. This equation is used in the last section in
the study of almost cosymplectic manifolds with K\"{a}hler fibres.
In particular, we prove that the fundamental 2-form of such kind of
manifolds is harmonic.

\section{Preliminaries}

As usual, all manifolds and maps involved are assumed to be of class
$C^\infty$. Furthermore, we denote by $\Gamma(E)$ the space of the
sections of a vector bundle $E$.

Let $M$ be a differentiable manifold. An almost contact structure on
$M$, denoted by $(\varphi, \xi, \eta)$, consists in a (1,1)-tensor
field $\varphi$, a non-singular vector field
 $\xi$, and a non-singular 1-form  $\eta$
that verify:
\begin{equation}\label{ec1}
\varphi^2=-I +\eta \otimes \xi
\end{equation}
and
\begin{equation}\label{ec2}
\eta(\xi)=1,
\end{equation}
where $I$ stands for the identity endomorphism of the fiber bundle
$TM$. It follows (see e.g. \cite{BL,CHN1}) that the manifold has odd
dimension and there holds
 $\varphi \xi=0$ and $\eta \circ \varphi=0$, as well. Moreover, $(M,\varphi,\xi,\eta)$ is called an almost contact manifold.

A Riemannian metric $g$ on $M$ is said to be adapted to the almost
contact structure $(\varphi,\xi,\eta)$ if
\begin{equation}\label{ec3}
g(\varphi X,\varphi Y)=g(X,Y) -\eta(X)\eta(Y)
\end{equation}
for any sections $X,Y \in \Gamma(TM)$. The setting
$(\varphi,\xi,\eta,g)$ is called an almost contact metric structure
on $M$, while $M$ is an almost contact metric manifold. It follows
from (\ref{ec1})-(\ref{ec3}) that the relation: $\eta(X)=g(X,\xi)$
holds for any section $X \in \Gamma(TM)$. Define next the
fundamental 2-form $\Phi$ of the almost contact metric manifold $M$
by the analogous to the almost Hermite geometry formula
\begin{equation}\label{ec4}
\Phi (X,Y)=g(X,\varphi Y)
\end{equation}
for any sections of the tangent bundle $X,Y \in \Gamma(TM)$. Basic
examples and features of those structures will be referred from the
book of Blair \cite{BL}.

Let $\pi:M \longrightarrow N$ be a surjective differentiable map
between two smooth manifolds $M$ and $N$. We call it a submersion if
its rank is constant, equal to the dimension of the target manifold
$N$.

Let now $M$ and $N$ be Riemannian manifolds with Riemannian metrics
$g$ and $g'$ respectively. Write $F_{x}=\pi^{-1}(x')$ for the fibre
in $x \in M$ of the map $\pi:M \longrightarrow N$, where
$\pi(x)=x'$. The tangent vectors to the fibre belong to the kernel
of the linear map $\pi_*=d\pi$; one calls them vertical. The
distribution of the vertical vectors, $\mathcal{V}$, is well defined
(with constant rank) and $\mathcal{V}=Ker\ d\pi$. It follows that
$\mathcal{V}$ is integrable and the (connected components of the)
fibres are its integral manifolds (of maximal rank). The orthogonal
complement of $\mathcal{V}$ in $TM$ with respect to the Riemannian
metric $g$ is the horizontal distribution, denoted by $\mathcal{H}$.
The submersion $\pi:(M,g) \longrightarrow (N,g')$ is a Riemannian
submersion if $d\pi_x: \mathcal{H}_x \longrightarrow T_{x'} N$
becomes a linear isometry, for any  $x \in M$.

If an horizontal vector field (a section in the horizontal
distribution $\mathcal{H}$) is $\pi$-related (so that projectable)
to a vector field on the base manifold $N$, one calls it basic. It
follows from the definition of the Riemannian submersion that
\begin{equation}\label{ec5}
g(X,Y)=g'(X',Y')\circ \pi
\end{equation}
if $X$ and $Y$ are basic vector fields, $\pi$-related to $X',Y' \in
\Gamma(TN)$, respectively.

The projections in $TM$ on $\mathcal{V}$ and $\mathcal{H}$ will be
respectively denoted by $v$, and $h$, as well. If $\nabla$ and
$\nabla'$ are the Levi-Civita connections of the Riemannian metrics
on $M$, $N$, then we have the following (see \cite{ON}).
\begin{prop}
Let $\pi : (M,g) \longrightarrow (N',g')$ be a Riemannian
submersion. If $X$, $Y$ are basic vector fields $\pi$-related to
$X'$, $Y'$ on $N$, then:\\
i. $h[X,Y]$ is basic, $\pi$-related to $[X',Y']$.\\
ii. $h(\nabla_X Y)$ is basic, $\pi$-related to $\nabla'_{X'} Y'$.\\
iii. $[X,V]$ is vertical for any $V \in \Gamma(V)$.
\end{prop}

The fundamental tensors of a submersion were defined by O'Neill.
They are $(1,2)$-tensors on $M$, given by the formulae:
\begin{equation}\label{ec6}
T_{\tilde X} \tilde Y=h\nabla_{v\tilde X} v\tilde Y +
v\nabla_{v\tilde X} h\tilde Y,
\end{equation}
\begin{equation}\label{ec7}
A_{\tilde X} \tilde Y=v\nabla_{h\tilde X} h\tilde Y +
h\nabla_{h\tilde X} v\tilde Y,
\end{equation}
where $\tilde X, \tilde Y$ are vector fields on $M$. One gets
immediately:
\begin{equation}\label{ec8}
T_U W = T_W U,\   \forall\ U,W \in \Gamma(V)
\end{equation}
and
\begin{equation}\label{ec9}
A_X Y =-A_Y X=\frac{1}{2}v[X,Y],\   \forall\ X,Y \in \Gamma(H).
\end{equation}

For the sake of simplicity, $U,V,W$ etc. will stand for vertical
vectors, $X,Y,Z$ etc. for horizontal vectors or vector fields,
respectively.

This paper will consider Riemannian submersions where the total
space $M$ is an almost contact metric manifold and the base space
$N$ is an almost Hermitian manifold with an almost complex structure
$J$ compatible to $g'$. Let $\Omega$ be the fundamental 2-form of
$(N,g',J)$ so that $\Omega(X',Y')=g(X',JY')$. Basic properties of
almost Hermitian manifolds are found e.g. in the fundamental work of
Kobayashi and Nomizu \cite{KN}.

Recall the definition of $(\varphi,J)$-holomorphic mappings from an
almost contact manifold  to an almost complex manifold (see
\cite{IP}).
\begin{definition}
Let $(M, \varphi, \xi, \eta)$ be almost contact and $(N,J)$ be
almost complex manifolds, respectively. The map $\pi:M
\longrightarrow N$ is $(\varphi,J)$-holomorphic if $J \circ
d\pi=d\pi \circ \varphi$.
\end{definition}
Now let consider Riemannian structures.
\begin{definition}
Let $(M, \varphi, \xi, \eta,g)$ be an almost contact  metric
manifold and $(N,J,g')$ be an almost Hermitian manifold. A
Riemannian submersion $\pi:M \longrightarrow N$ is a contact-complex
(Riemannian submersion) if it is $(\varphi,J)$-holomorphic, as well.
\end{definition}

Next, we recall some examples (see \cite{FIP}).
\begin{ex}
For $(M,g_1)$ and $(N,g_2)$ Riemannian manifolds, take the product
$M \times N$. Its tangent bundle naturally splits as $T(M \times
N)=TM \oplus TN$. For $(x,y) \in M \times N$ one defines the
symmetric bilinear forms $g_1 \oplus g_2$ on all $T_{(x,y)}(M\times
N)=T_x M \oplus T_y N$ as
\begin{equation}\label{ec11}
g=\pi^*_1 g_1+\pi^*_2 g_2
\end{equation}
where $\pi_1:M \times N\longrightarrow M$ and $\pi_2:M \times N
\longrightarrow N$ are the natural projections. One gets a
Riemannian tensor on $M \times N$ and a Riemannian submersion from
$M \times N$ to $M$ (the mapping $\pi_1$).
\end{ex}

\begin{ex}
R. Bishop and B. O'Neill introduced the warped product of Riemannian
manifolds, which varies the metric on some factor, e.g. $\tilde
g_2=f(x)g_2$, where $f$ is a positive function on $N$ and
\[g=\pi^*_1 g_1+(f\circ \pi_1)\pi^*_2 g_2.\] The metric on the fibres
depends upon the value of the map in the base point of them. It
holds that $\pi_1:(N \times_f F, \tilde g)\longrightarrow (N,g_1)$
is a Riemannian submersion with fundamental tensor
\begin{equation}\label{ec12}
T_U V=-\frac{1}{2f}g(U,V){\rm grad}f.
\end{equation}

The above formula shows that the fibres become totally umbilical
submanifolds of $M \times N$. The associated mean curvature vector
field $H=-\frac{1}{2f} {\rm grad} f$ shows further that the fibres
cannot be minimal (equivalently, since umbilical, nor geodesic)
unless $f$ is constant. If $f$ is constant, one essentially gets
again the product manifold of the above example, where both
projections are Riemannian submersions with totally geodesic fibres.
\end{ex}

\begin{ex}
Take as $\pi$ the natural projection of the tangent bundle of a
Riemannian manifold $(M,g)$ to it. Then $T$M can be endowed with the
canonical metric (Sasakian metric) $G$ by: \[G(\tilde X,\tilde
Y)=g(\pi_*\tilde X,\pi_*\tilde Y)\circ \pi+g(K\tilde X,K\tilde
Y)\circ \pi\] where $K$ is the Dombrowski connection map associated
to $\nabla^{M}$. One immediately gets the Riemannian submersion
definition properties for $\pi:(TM,G)\longrightarrow (M,g)$ . The
fibres are totally geodesic. If $TM$ is endowed with the
Cheeger-Gromoll metric (see \cite{BL}) one gets a Riemannian
submersion from $(TM,\tilde g)$ to $(M,g)$, as well.
\end{ex}

\begin{ex}
Taking above the hypersurface consisting of unit vectors in $TM$
with respect to the metric $g$, one gets that $T_1 M=\{v \in TM \mid
\parallel v\parallel =1\}$ is the total space (of a spherical bundle
over the base manifold $M$) with the almost contact metric structure
naturally induced by the almost Hermitian structure on $TM$ and the
projection $\pi:T_1 M \longrightarrow M$ is a contact-complex
Riemannian submersion \cite{BL}.
\end{ex}

\begin{ex}
The Hopf fibration $\pi:S^{2n+1} \longrightarrow CP^n$ gives a
contact-complex Riemannian submersion considering the Sasakian
structure on $S^{2n+1}$ and a suitable multiple of the Fubini-Study
metric on the complex manifold $CP^n$.
\end{ex}

\section{Tensor $B$ of contact-complex Riemann submersions}

Define on the total space a $(1,2)$-tensor $B$ in the
contact-complex Riemannian submersions setting by
\begin{equation}\label{2}
B(\tilde X,\tilde Y)=v \nabla_{h \tilde X} \varphi h \tilde Y-v
\nabla_{\varphi h \tilde X} h \tilde Y+ h \nabla_{h \tilde X}
\varphi v \tilde Y-h \nabla_{\varphi h \tilde X} v \tilde Y,
\end{equation}
where $\nabla$ is the Levi-Civita connection on $(M,g)$.

This tensor appears in the work of Watson and Vanhecke \cite{WV1} on
Riemannian submersions between almost Hermitian manifolds.
\begin{prop}
Let $\pi:M \longrightarrow N$ be a contact-complex Riemannian
submersion. Then\\
i. $B(V,\tilde Y)=0$ for vertical $V$ and arbitrary $\tilde Y$. In
particular, $B$ restricts to zero on the vertical distribution.\\
ii. $B(\varphi X,\xi)=h\nabla_X \xi$ for any horizontal $X$. In
particular, if $M$ is Sasakian, then $B(\varphi X,\xi)$ is not zero.
If $M$ is almost cosymplectic, then $B(\varphi X,\xi)=0$.\\
iii. $B$ restricts to a $J$-invariant tensor on the horizontal
distribution $\mathcal{H}$.
\end{prop}
\begin{proof}
i. The first assertion is trivial.\\
ii. From (\ref{2}) one gets \[B(\varphi X,\xi)=-h \nabla_{\varphi h
\varphi X}\xi=-h\nabla_{\varphi^2 X}\xi=h\nabla_X \xi\] since
$\varphi ^2 X=-X$, for any horizontal $X$. If $M$ is Sasaki, then
$\nabla_X \xi=-\varphi X$, so $B(\varphi X,\xi)\not=0$ for any
horizontal vector field $X$.

If $M$ is cosymplectic we have $\nabla_X \xi=0$ and the conclusion
is obvious.\\
iii. Since $\mathcal{H}$ is $\varphi$-invariant and $\varphi ^2
X=-X$ for any horizontal vector field $X$, the result follows.
\end{proof}

\begin{prop}\label{p2}
Let $\pi:M \longrightarrow N$ be a  contact-complex Riemannian
submersion. Then:\\
i. For any vector field $\tilde Z$ on $M$, the linear mapping
$B(\tilde Z,\cdot )$ sends forward and backwards horizontal vectors
to vertical vectors.\\
ii. $B(X,Y)=B(Y,X)=A_X \varphi Y - A_{\varphi X}Y$ for horizontal
vector fields $X$ and $Y$.\\
iii. $B(X,Y)=0$ for horizontal X and Y if and only if $B(X,X)=0$ for
any horizontal $X$.
\end{prop}
\begin{proof}
i,ii. Trivial; iii. This statement follows immediately from the
symmetry of $B$ on the horizontal distribution $\mathcal{H}$.
\end{proof}

\begin{prop}
Let $\pi:M \longrightarrow N$ be a contact-complex Riemannian
submersion. Then for any basic X, Y and vertical V one has
\begin{equation}\label{ex}
g(B(X,Y),V)=g(\nabla_V \varphi Y+\varphi \nabla_V Y,X)+2Vg(\varphi
X,Y)
\end{equation}
\end{prop}
\begin{proof}
Using the properties of  $A$ and $\varphi$ one gets \[g(A_X \varphi
Y,V)=g(v\nabla_X \varphi Y,V)=-g(\varphi Y,\nabla_X V)=\]
\[=-g(\varphi Y, [X,V]+\nabla_V X)=-g(\varphi Y,\nabla_V X)=\]
\[=-Vg(X,\varphi Y)+g(X,\nabla_V \varphi Y)\] since the bracket
$[X,V]$ is vertical. Analogously \[g(\nabla_{\varphi X}
Y,V)=-Vg(\varphi X,Y)-g(\varphi \nabla_V Y,X)\] so that with
property ii. in Proposition \ref{p2} the result follows.
\end{proof}

Next, we recall the fundamental properties of Riemann
contact-complex submersions due to Watson \cite{WAT}.

\begin{prop}\label{3.2} \cite{WAT}
Let $\pi :M \rightarrow N$ be a contact-complex Riemannian submersion. Then: \\
i. $ \xi \in \Gamma (\mathcal{V}), \eta (X) =0$, for any horizontal
$X$; the vertical and horizontal distributions $\mathcal{V}$
and $\mathcal{H}$ are $ \varphi $-invariant. \\
ii. $h(\nabla _{X} \varphi) Y$ is basic, $ \pi$-related to $(\nabla '_{X'} J) Y'$, for basic $X,Y$ on $M$, $ \pi$-related to $X',Y'$ on $N$. \\
iii. $\pi ^{*} \Omega=\Phi$.  \\
iv. $N_1 (X,Y)=(\pi ^{*} N_{J})(X,Y)$, for basic $X,Y$, where $N_1$
denotes the normality tensor of $M$ and $N_{J}$ the Nijenhuis tensor
of $N$.
\end{prop}

Next, we investigate the contact-complex submersions satisfying
Gray-type curvature conditions, denoted $K_{i \varphi}$, $i=1,2,3$.
We recall that the so-called Gray identities were introduced by A.
Gray \cite{GR} for almost Hermitian manifolds $(M,g,J)$:
\[
K_1:\ R(X,Y,Z,W)=R(X,Y,JZ,JW),
\]
\[
K_2:\ R(X,Y,Z,W)=R(JX,Y,Z,JW)+R(X,JY,Z,JW)+R(X,Y,JZ,JW),
\]
\[
K_3:\ R(X,Y,Z,W)=R(JX,JY,JZ,JW).
\]

For an almost contact metric manifold $(M, \varphi, \xi, \eta,g)$,
the analogous identities have been considered by Bonome, Hervella
and Rozas \cite{BHR}:
\[
K_{1\varphi}:\ R(X,Y,Z,W)=R(X,Y,\varphi Z,\varphi W),
\]
\[
K_{2\varphi}:\ R(X,Y,Z,W)=R(\varphi X,Y,Z,\varphi W)+R(X,\varphi
Y,Z,\varphi W)+R(X,Y,\varphi Z,\varphi W),
\]
\[
K_{3\varphi}:\ R(X,Y,Z,W)=R(\varphi X,\varphi Y,\varphi Z,\varphi
W).
\]

We remark that other Gray-type curvature conditions on almost
contact metric manifolds have been introduced in \cite{MM}.

\begin{prop}\label{3.3}
Let $\pi :M \rightarrow N$ a contact-complex submersion where the total space $M$ satisfies $K_{1 \varphi}$. Then: \\
i. If O'Neill tensor $T$ is $\varphi$-bilinear, then the fibres satisfy on the $K_{1 \varphi}$. \\
ii. If O'Neill tensor $A$ vanishes, then the base space $N$
satisfies on $K_{1}$.
\end{prop}
\begin{proof}
i. The curvature relations in a Riemannian submersion setting give:
\[R(U,V,F,W)= \hat{R} (U,V,F,W)+g(T_{U} W, T_{V} F)-g(T_{V} W, T_{U}
F),\] \[R(U,V,\varphi F,\varphi W)=\hat{R} (U,V,\varphi F,\varphi
W)+g(T_{U} \varphi W, T_{V} \varphi F)-g(T_{V} \varphi W, T_{U}
\varphi F),\] where $R$ and $\hat R$ are the curvature tensors on
$M$ and one of the fibres, respectively,  $U,V,F,W$ being vertical
vector fields.

Then the assertion follows using the $\varphi$-invariance of the
metric $g$.\\
ii. From \[R(X,Y,Z,H)=R^*(X,Y,Z,H)-2g(A_X Y,A_Z H)+g(A_Y Z,A_X
H)-g(A_X Z,A_Y H),\] where $R^*(X,Y,Z,H)=R'(X',Y',Z',H')\circ \pi$
and $X,Y,Z,H$ are basic vector fields $\pi$-related to $X',Y',Z',H'$
respectively, it follows that \[R(X,Y,\varphi Z,\varphi
H)=R^*(X,Y,\varphi Z,\varphi H)-2g(A_X Y,A_{\varphi Z} \varphi H)+\]
\[+g(A_Y \varphi Z,A_X \varphi H)-g(A_X \varphi Z,A_Y \varphi H).\]
If $A\equiv 0$, then \[R'(X',Y',Z',H')=R'(X',Y',JZ,JH)\] since
$\varphi Z$, $\varphi H$ are related to $JZ'$ and $JH'$.
\end{proof}

\begin{prop}
Let $\pi:M \longrightarrow N$ be a contact-complex Riemannian
submersion. Suppose that both $M$ and $N$ satisfy $K_{1 \varphi}$
and $K_1$ respectively. If $[X,\varphi X] \in \Gamma(V)$ for any
basic vector field $X$, then $\mathcal{H}$ is integrable.
\end{prop}
\begin{proof}
The curvature equations give now, setting $Z=X$ and $H=Y$:
\[-2g(A_X Y,A_X Y)+g(A_Y X,A_X Y)+g(A_X X, A_Y Y)+\]
\[+2g(A_X Y,A_X Y)-g(A_Y \varphi X, A_X \varphi Y)-g(A_{\varphi X} X,
A_Y \varphi Y)=0.\] Because $[X,\varphi X] \in \Gamma(V)$ says that
$B(X,X)=0$ for an horizontal vector field $X$, it is now known that
$B(X,Y)=0$ for horizontal $X,Y$.

It easy to see that $A_{\varphi X} \varphi Y=-A_X Y$ and this gives
now \[5g(A_X Y,A_X Y)+g(A_X \varphi Y, A_X \varphi Y)=0.\] Both
summands are non-negative, so that vanish; in particular $\mathcal
H$ is involutive.
\end{proof}

\begin{prop}
Let $\pi:M\rightarrow N$ be a contact-complex Riemannian submersion.
Suppose that $M$ satisfies on the $K_{2\varphi}$. If $A$ is
$\varphi$-bilinear then the base space $N$ also satisfies on $K_2$.
\end{prop}
\begin{proof}
The fundamental equations of a submersion give, taking into account
the hypothesis, that
\[
R'(X',Y',Z',T')-R'(JX',JY',Z',T')-R'(JX',Y',JZ',T')-R'(JX',Y',Z',JT')=
\]
\[
=2g(A_XY,A_ZT)-g(A_YZ,A_XT)-g(A_ZX,A_YT)
\]
\[
+g(A_{\varphi Y}Z,A_{\varphi X}T)-2g(A_{\varphi X}\varphi
Y,A_ZT)+g(A_Z\varphi X,A_{\varphi Y}T)
\]
\[
+g(A_Y\varphi Z,A_{\varphi X}T)-2g(A_{\varphi X} Y,A_{\varphi
Z}T)+g(A_{\varphi Z}\varphi X,A_YT)
\]
\[
+g(A_YZ,A_{\varphi X}\varphi T)-2g(A_{\varphi X} Y,A_Z{\varphi
T})+g(A_Z\varphi X,A_Y\varphi T).
\]

The right-hand side vanishes as follows. We have
\[
g(A_{\varphi Y}Z,A_{\varphi X}T)=g(\varphi A_YZ,\varphi
A_XT)=g(A_YZ,A_XT)+\eta(A_YZ)\eta(A_XT),
\]
since $A$ is $\varphi$-bilinear. On the other side,
\[
\eta(A_YZ)=g(A_YZ,\xi)=g(A_Y\varphi P,\xi)=g(\varphi A_YP,\xi)=0
\]
taking into account that $\mathcal{H}$ is $\varphi$-invariant, as
well, so that it exists $P\in\Gamma(\mathcal{H})$ such that
$Z=\varphi P$. Then $g(A_{\varphi Y}Z,A_{\varphi X}T)=g(A_YZ,A_XT)$.

Now consider $-2g(A_{\varphi X}\varphi Y,A_ZT)$ which can be
expanded
\[
g(A_{\varphi X}\varphi Y,A_ZT)=g(\varphi^2 A_XY,A_ZT)=
\]
\[
=-g(A_XY,A_ZT)+\eta(A_XY)g(\xi,A_ZT)=-g(A_XY,A_ZT).
\]

In an analogous way, all terms containing $\varphi$ will be replaced
by quantities not involving $\varphi$ so that
\[
R'(X',Y',Z',T')=R'(JX',JY',Z',T')+R'(JX',Y',JZ',T')+R'(JX',Y',Z',JT')
\]
for any vector fields $X',Y',Z',T'$ on the base space $N$.
\end{proof}

One gets in an analogous way the following result.
\begin{prop}
Let $\pi:M\rightarrow N$ be a contact-complex Riemannian submersion.
Suppose that  the O'Neill tensor $A$ satisfies on $A_X\varphi
Y=A_{\varphi X}Y$, for any horizontal vector fields $X,Y$. If
$K_{3\varphi}$ holds on  $M$, then $K_3$ holds on
 $N$.
\end{prop}

\section{The structure equation of a contact-complex Riemannian submersion}

The relationship between the codifferentials of the basic 2-forms  $
\Phi $, $ \hat{\Phi }$ and $ \Omega $ of the total space, a fibre
and the base space of a contact-complex Riemannian submersion $\pi:M
\rightarrow N$ will be established in what follows.

The codifferential of $\Phi $ on $ M $ is
\begin{eqnarray}\label{(3.1)}
 \delta \Phi(\tilde{X}) &=& -\sum_{i=1}^{n} \{ (\nabla_{X_{i}}\Phi)(X_{i}, \tilde{X}) + (\nabla_{\varphi X_{i}}\Phi)(\varphi X_{i}, \tilde{X}) \}\nonumber \\
&-&\sum_{j=1}^{m-n} \{ (\nabla_{V_{j}}\Phi)(V_{j}, \tilde{X}) +
(\nabla_{\varphi V_{j}}\Phi)(\varphi V_{j}, \tilde{X}) \} -
(\nabla_{\xi}\Phi)(\xi, \tilde{X})
\end{eqnarray}
with \[\{(X_{i},\varphi X_{i}); (V_{i},\varphi V_{i}, \xi) \}, 1\leq
i \leq n, 1 \leq j \leq m-n,\] a $\varphi$ - orthonormal local basis
in the tangent space, such that $(X_{i},\varphi X_{i})$ are
 basic and $ (V_{i},\varphi V_{i})$ are vertical. Recall that $\varphi
X_{i}, \varphi V_{i}$ are horizontal and vertical respectively,
since both distributions $ \mathcal{H}$ , $\mathcal{V}$ are
$\varphi$-invariant, and the Reeb vector field $\xi$ is vertical.

\begin{thrm}
\noindent {\rm \textbf{(The structure equation)}} Let $\pi:M
\rightarrow N$ be a contact-complex Riemannian submersion. The
following formula holds good:
\begin{equation}\label{codif2}
\delta \Phi(\tilde{X})=\delta'\Omega(X')+\hat{\delta}
\hat{\Phi}(V)+g(H,\varphi X)+\frac{1}{2}g(Tr B^{h}, V)
\end{equation}
where $\delta, \hat{\delta}, \delta' $  are the codifferential on
$M$, $N$ and a fibre (by respect to the induced metric). Here $H$ is
the mean curvature vector field of a fibre and $Tr B^{h} $ stands
for the trace of the restriction of the tensor $B$ to the
distribution $\mathcal{H}$.
\end{thrm}
\begin{proof}
If $\tilde{X}$ is an arbitrary vector field on $M$, then there is an
unique decomposition $\tilde{X}= X+V$ where $X$ is horizontal and
$V$ is vertical, so that $ X=h \tilde{X}$ and $ V=v \tilde{X}$. One
calculates next $\delta\Phi(X)$ and $\delta\Phi(V)$. Using the
formula
\[(\nabla_{\tilde{X}}\Phi)(\tilde{Y},\tilde{Z})=-g((\nabla_{\tilde{X}}\varphi)\tilde{Y},
\tilde{Z})\] it follows from (\ref{(3.1)}) that
\begin{eqnarray*}
\delta \Phi(X) &=& \sum_{i=1}^{n} \{ g((\nabla_{X_{i}}\varphi)X_{i}, X) + g((\nabla_{\varphi X_{i}}\varphi)\varphi X_{i}, X) \}\nonumber \\
&&+\sum_{j=1}^{m-n} \{ g((\nabla_{V_{j}}\varphi)V_{j}, X) +
g((\nabla_{\varphi V_{j}}\varphi)\varphi V_{j}, X) \} + g(
(\nabla_{\xi}\varphi)\xi, X).
\end{eqnarray*}

The generic term in the first summation can be written as
\[g(\nabla_{X_{i}}\varphi X_{i} - \varphi \nabla_{X_{i}} X_{i} ,
X)-g( \nabla_{\varphi X_{i}} X_{i}, X)-g(\varphi \nabla_{\varphi
X_{i}} \varphi X_{i}, X)
\]
or
\[g((h \nabla_{X_{i}}\varphi X_{i}, X) - g(h \varphi
\nabla_{X_{i}} X_{i} , X)-g( h \nabla_{\varphi X_{i}} X_{i}, X)-g(h
\varphi \nabla_{\varphi X_{i}} \varphi X_{i}, X).
\]

Considering now basic, local vector fields $X_{i}, \varphi X_{i}, X$
on $M$, $\pi$-related to $ X'_{i}, \varphi X'_{i}, X' $ on $N$, we
obtain
\[
g'((\nabla'_{X'_{i}}J X'_{i}, X') - g'(J \nabla'_{X'_{i}} X'_{i} ,
X')-g'( \nabla'_{J X'_{i}} X'_{i}, X')-g'(J \nabla'_{J X'_{i}} J
X'_{i}, X').
\]

Therefore the first sum in the expression of $\delta \Phi(X) $ can
be written as
\[ \sum_{i=1}^{n}\{ g'((\nabla'_{X'_{i}}J )X'_{i},
X')o \pi + g'((\nabla'_{J X'_{i}}J) J X'_{i}, X')o \pi \}\] or,
equivalently, $\delta'\Omega (X')o \pi,$ where $\Omega$ is the
fundamental form on the base space $ N$.

Next, the generic summand in the expression of $\delta \Phi(X) $
(the second sumation) is \[g(\nabla_{V_{j}}\varphi V_{j} - \varphi
\nabla_{ V_{j}}V_{j}, X) - g(\nabla_{\varphi V_{j}} V_{j}+ \varphi
\nabla_{\varphi V_{j}} \varphi V_{j}, X) \] or
\[g(\nabla_{V_{j}}\varphi V_{j} -
\nabla_{ \varphi V_{j}}V_{j}, X) - g(\varphi (\nabla_{ V_{j}} V_{j}+
\nabla_{\varphi V_{j}} \varphi V_{j}), X).
\]

The first term above vanishes as $g([ V_{j}, \varphi V_{j}], X)=0$,
since $[ V_{j}, \varphi V_{j}]$ is tangent to the fibres. Hence the
second sum in the expression of $\delta \Phi(X) $ can be written as
\[ \sum_{i=1}^{n}g(\nabla_{ V_{j}} V_{j}+
\nabla_{\varphi V_{j}} \varphi V_{j},\varphi  X),\] or,
equivalently, $g(H, \varphi X)-g(\nabla_{\xi}\xi,\varphi X)$, where
$H$ is the mean curvature vector field of the fibres.

One finally gets therefore
\begin{equation}\label{codif3}
 \qquad \delta \Phi(X)= g(H, \varphi X) + \delta' \Omega (X') o
 \pi.
\end{equation}

On another hand, we have:
\begin{eqnarray*} \qquad \delta \Phi(V) &=& \sum_{i=1}^{n} \{ g((\nabla_{X_{i}}\varphi)X_{i}, V) + g((\nabla_{\varphi X_{i}}\varphi)\varphi X_{i}, V) \}\nonumber \\
&+&\sum_{j=1}^{m-n} \{ g((\nabla_{V_{j}}\varphi)V_{j}, V) +
g((\nabla_{\varphi V_{j}}\varphi)\varphi V_{j}, X) \} + g(
(\nabla_{\xi}\varphi)\xi, V).
\end{eqnarray*}

For the generic summand in the first summation we obtain
\[g((v \nabla_{X_{i}}\varphi X_{i} - v \varphi \nabla_{X_{i}} X_{i} ,
V)-g( v \nabla_{\varphi X_{i}} X_{i}+ v \varphi \nabla_{\varphi
X_{i}} \varphi X_{i}, V)
\]
\[=g( A_{X_{i}}\varphi X_{i} - A_{\varphi X_{i}} X_{i} , V)- g(
\varphi A_{ X_{i}} X_{i}+ \varphi A_{\varphi X_{i}} \varphi X_{i},
V)\] by the definition of $A$ and the fact that $\varphi$ and the
projection $v$ commute. The skew-symmetry of $A$ gives the vanishing
of the last parenthesis.

Denoting by $ B^{h}$ the restriction of $B$ to the horizontal
distribution, it follows that \[g( A_{X_{i}}\varphi X_{i} -
A_{\varphi X_{i}} X_{i} , V)\] is the generic summand in
 $ \frac{1}{2}g(Tr B^{h}, V)$, where $Tr B^{h}$ is the trace of
$ B^{h}$. Compute now the second sum in $\delta \Phi(V) $ as the
restriction of $\delta \Phi(V) $ to the vertical distribution
$\mathcal{V}$. Distinguish by a hat the induced objects on fibres
(which are $\varphi$-invariant submanifolds). It follows that
\begin{eqnarray*}
\hat{\delta }\hat{\Phi}(V) = -\sum_{j=1}^{m-n} \{
(\hat{\nabla}_{V_{j}}\hat{\Phi})(V_{j}, V) + (\hat{\nabla}_{\varphi
V_{j}}\hat{\Phi})(\varphi V_{j}, V) \} -
(\hat{\nabla}_{\xi}\hat{\Phi})(\xi, V).
\end{eqnarray*}

The first summand expands
\begin{eqnarray*}(\hat{\nabla}_{V_{j}}\hat{\Phi})(V_{j}, V)&=& -g((\hat{\nabla}_{V_{j}}\varphi)V_{j}, V)
=-g(\hat{\nabla}_{V_{j}}\varphi V_{j}, V)+g(\varphi \hat{\nabla}_{V_{j}}V_{j}, V)\nonumber \\
&=&-g(\nabla_{V_{j}}\varphi V_{j}-\varphi \nabla_{V_{j}}V_{j}, V)
=-g((\nabla_{V_{j}}\varphi ) V_{j}, V)\nonumber \\
&=& (\nabla_{V_{j}}\Phi )(V_{j}, V).
\end{eqnarray*}

In an analogous way, it follows that $\delta \Phi(V) $ restricts to
the fibres to  $\hat{\delta} \hat{\Phi}(V) $.  Therefore
\begin{equation}\label{codif4}
 \qquad  \delta \Phi(V)=\hat{\delta} \hat{\Phi}(V) + \frac{1}{2}g(Tr B^{h}, V)
 \end{equation}

From (\ref{codif3}) and  (\ref{codif4}) one gets (\ref{codif2}),
where $\delta \Phi(\tilde{X})=\delta \Phi(X)+\delta \Phi(V).$
\end{proof}

In what follows this formula will be used in the study of almost
cosymplectic manifolds with K\"{a}hler fibres introduced and studied
by Z. Olszak \cite{OLS}.

\section{Contact-complex Riemannian
submersions from almost cosymplectic manifolds with K\"{a}hler
leaves}

Recall that an  almost contact metric manifold $M$ is almost
cosymplectic if $\eta$ and the fundamental form
 $\Phi$ are closed, i.e.
\begin{equation}\label{e-1}
d\Phi=0, d\eta=0.
 \end{equation}

The identity $d\eta=0$ shows that the distribution $D=\{X\in TM
\mid\eta(X)=0\}$ is integrable and its (maximal) integral manifolds
are hypersurfaces in $M$. The restrictions of $\Phi$  and $\eta$ to
the associated foliation are closed forms, so that any leave  is an
almost K\"{a}hler submanifold.

Denote by $A^{*}$ the endomorphism of the tangent space given by
\begin{equation}\label{e0}
 \qquad \qquad  A^{*} X=-\nabla_{X} \xi,\ \forall X \in \Gamma (TM).
\end{equation}

\begin{prop} \cite{OLS} The tensor $A^{*}$ satisfies on the following formulae
\begin{equation}\label{e1}
 g(A^{*} X, Y)= g( X, A^{*}Y),
 \end{equation}
 \begin{equation}\label{e2}
A^{* }\varphi+\varphi A^{*} =0,\ A^{*}\xi =0,\ \eta o A^{*} =0,
\end{equation}
\begin{equation}\label{e3}
(\nabla_{X}\varphi) Y=- g(\varphi A^{*} X, Y)\xi +\eta(Y)\varphi
A^{*}X.
\end{equation}
\end{prop}

It is trivial that a cosymplectic manifold is almost cosymplectic
with K\"{a}hler leaves. Also, the tensor field $ A^{*}$ vanishes in
this case. Olszak  gave examples of almost cosymplectic manifolds
with K\"{a}hler leaves that are not cosymplectic (see \cite{OLS}).
He proved as well that if an almost cosymplectic manifold with
K\"{a}hler leaves has vanishing tensor $ A^{*} $, then it is
actually cosymplectic.

We recall that an almost contact structure $(\phi,\xi,\eta)$ is
normal if and only if the next four tensors $N_1$, $N_{2}$, $N_{3}$
and $N_{4}$ vanish (see \cite{BL}):
\[
N_1(X,Y)=[\phi,\phi](X,Y)+2d\eta(X,Y)\xi,
\]
\[
N_{2}(X, Y)=(L_{\varphi X}\eta)Y-(L_{\varphi Y}\eta)X,
\]
\[
N_{3}(X)=(L_{\xi}\varphi)X,
\]
\[
N_{4}(X)=(L_{\xi}\eta)X,
\]
where $L_Z$ denotes the Lie derivative with respect to $Z$. It is
known that the vanishing of $N_{1}$ on contact metric manifolds
gives the vanishing of $N_{2}$, $N_{3}$ and $N_{4}$, too. The
converse is not true, in general (see \cite{BL} for details). The
almost cosymplectic manifold with  K\"{a}hler leaves have,
generally, nonvanishing tensors $N_{1}$ and $N_{3}$, but  $N_{2}$
and $N_{4}$  are all zero, as we can see in what follows.

\begin{prop}
 Let $M$  be an almost cosymplectic manifold with  K\"{a}hler leaves. Then $N_{2}$ and $N_{4}$
 vanish. Moreover $N_{3}$ vanishes if and only if $M$ is a cosymplectic manifold.
\end{prop}
\begin{proof}
The vanishing of $N_{2}$ and $N_{4}$ follows easily using
(\ref{e1}), (\ref{e2}) and (\ref{e3}).

On the other hand, by direct computation we derive
\begin{eqnarray*}  N_{3}(X)&=&(L_{\xi}\varphi)X=[\xi,\varphi X]-\varphi[\xi, X] \\
&=&\nabla_{\xi}\varphi X-\nabla_{\varphi X}\xi-\varphi \nabla_{\xi} X+\varphi \nabla_{X} \xi \\
&=&A^{*} \varphi X-\varphi A^{*}X \\
&=&2A^{*} \varphi X
\end{eqnarray*}
and the conclusion follows.
\end{proof}

\begin{prop} Let $M(\varphi,\xi,\eta,g)$ be an almost cosymplectic manifold with  K\"{a}hler leaves. Then its fundamental form
is harmonic.
\end{prop}
\begin{proof}
Since $\Phi$ is closed, it remains to prove that it is co-closed, as
well.

Let  $\{e_{i},\varphi e_{i},\xi\}$ be a local $\varphi$ - Hermitian
frame on $M$, $(1 \leq i \leq m).$ It follows that
\begin{eqnarray*} \delta \Phi(\tilde{X}) &=& -\sum_{i=1}^{m}  (\nabla_{e_{i}}\Phi)(e_{i}, \tilde{X})
-\sum_{i=1}^{m}  (\nabla_{\varphi e_{i}}\Phi)(\varphi e_{i}, \tilde{X})-(\nabla_{\xi}\Phi)(\xi, \tilde{X})
\end{eqnarray*}
which can be written
\begin{eqnarray*} \delta \Phi(\tilde{X}) &=& \sum_{i=1}^{m} g( (\nabla_{e_{i}}\varphi)e_{i}, \tilde{X})
+ \sum_{i=1}^{m}  g((\nabla_{\varphi e_{i}}\varphi)\varphi e_{i},
\tilde{X})+g((\nabla_{\xi}\varphi)\xi, \tilde{X}).
\end{eqnarray*}

The last summand is obviously zero. Taking into account (\ref{e3}),
the generic summand in the first summation takes the form
\[ g(
(\nabla_{e_{i}}\varphi)e_{i}, \tilde{X})=-g(\varphi
A^{*}(e_{i}),e_{i})g(\xi,\tilde{X}).
\]

The second term is \[ g((\nabla_{\varphi e_{i}}\varphi)\varphi
e_{i}, \tilde{X})=-g(\varphi A^{*}(\varphi e_{i}),\varphi
e_{i})g(\xi,\tilde{X}).\]

Using the properties of $A^{*}$, it follows that  $\delta
\Phi(\tilde{X})=0,\ \forall \tilde{X} \in \Gamma(TM).$ Finally
\[\triangle \Phi=(d \delta+\delta d)\Phi=0\]
and the result follows.
\end{proof}

\begin{thrm}
Let $\pi:M \rightarrow N$ be a contact-complex Riemannian
submersion. Suppose that $M$ is almost cosymplectic with K\"{a}hler
leaves. Then\\
i. The base space $N$ is K\"{a}hler if and only if the O'Neill
tensor $A$ satisfies $A_{X}\xi=0$ for any $X\in\mathcal{H}$.\\
ii. The fundamental form $\Omega$ of the base is harmonic if and
only if $Tr B^{h}=0.$
\end{thrm}
\begin{proof}
i. From Proposition \ref{3.2} one has
\[(\pi^{*}N_{j})(X,Y)=N_{1}(X,Y)\]
for any basic vector fields $X,Y$ on $M$.

One shows next that $ N_{1}(X,Y)=0$ for any basic vector fields
$X,Y$. It is known that the vanishing of  $ N_{1}$ is equivalent to
the identity \[(\nabla_{X}\varphi)Y -(\nabla_{\varphi
X}\varphi)\varphi Y + \eta(Y) \nabla_{\varphi X}\xi=0.\]

But the last term vanishes as $Y$ is basic. Using (\ref{e3}), the
last equation becomes \[g(\varphi \nabla_{X}\xi, Y)- g(\varphi
\nabla_{\varphi X}\xi,\varphi Y)=0\] or, equivalently, $(
\nabla_{X}\xi, \varphi Y) =0.$ Taking into account that $A_{X}\xi=h
\nabla_{X}\xi$ and the $\varphi$-invariance of $\mathcal{H}$, the
assertion follows as
\[N_{1}(X,Y)=N_{J}(X',Y'),\] with $X, Y$ basic vector fields on $M$, $\pi$-related to $X',Y'$ on
$N$.\\
ii. It is known that a $\varphi$-invariant submanifold of an almost
cosymplectic manifold tangent to the characteristic vector field
$\xi$ is minimal, so that the fibres are minimal: $H=0$. The
restriction to the fibres of $\delta$ is $\hat{\delta}$, so that
$\hat{\delta}\hat{\Phi} = 0$.

Since $\delta \Phi = 0$, the assertion follows from the structure
equation (\ref{codif2}).
\end{proof}

\proof[Acknowledgements] S.I. and G.E.V. are partially supported by
CNCSIS - UEFISCSU, project PNII - IDEI code 8/2008, contract no.
525/2009. R.M. acknowledges that this work was partially supported
by CNCSIS - UEFISCSU, project PNII - IDEI code 1193/2008, contract
no. 529/2009.

\end{document}